\documentclass{article}

\usepackage{amsmath,amsfonts,amsthm,amssymb}
\usepackage{url}

\theoremstyle{plain}
\newtheorem{theorem}{Theorem}
\newtheorem{corollary}[theorem]{Corollary}
\newtheorem{proposition}[theorem]{Proposition}

\theoremstyle{definition}
\newtheorem{remark}[theorem]{Remark}
\newtheorem{example}[theorem]{Example}
\newtheorem{definition}[theorem]{Definition}

\newcommand{\R}{\mathbb{R}}
\newcommand{\N}{\mathbb{N}}
\newcommand{\Z}{\mathbb{Z}}
\newcommand{\T}{\mathbb{T}}

\begin{document}

\title{Inequalities and majorisations for the Riemann--Stieltjes
integral on time scales\thanks{Submitted 30-Apr-2009;
revised 15-Feb-2010; accepted 24-Mar-2010;
for publication in Math. Inequal. Appl.}}

\author{Dorota Mozyrska$^1$\\
        \url{d.mozyrska@pb.edu.pl}
        \and
        Ewa Paw\l uszewicz$^{1,2}$\\
        \url{ewa@ua.pt}
        \and
        Delfim F. M. Torres$^2$\\
		\url{delfim@ua.pt}}

\date{$^{1}$Faculty of Computer Science\\
        Bia{\l}ystok University of Technology\\
		15-351 Bia\l ystok, Poland\\[0.3cm]
      $^{2}$Department of Mathematics\\
		University of Aveiro\\
		3810-193 Aveiro, Portugal}

\maketitle


\begin{abstract}
We prove dynamic inequalities of majorisation type
for functions on time scales. The results
are obtained using the notion
of Riemann--Stieltjes delta integral and
give a generalization of
[App. Math. Let. {\bf 22} (2009), no.~3, 416--421]
to time scales.

\bigskip


\medskip


\end{abstract}


\section{Introduction}

In the literature one can find many results known as
\emph{Majorisation Theorems}. In the recent papers
\cite{BCD,dragomir} inequalities of majorisation type for convex functions
and Stieltjes integrals are given.
The main goal of the present note is
to unify and generalize such discrete-time
and continuous-time inequalities by means of the notion
of Riemann--Stieltjes integral on time scales \cite{MPT,Rev[1]}.

The theory and applications of delta derivatives and integrals
on time scales is a relatively new area that
is receiving an increase of interest and attention \cite{book:ts1}.
The concept of Riemann--Stieltjes integration on time scales
was introduced in 1992 by S.~Sailer \cite{Rev[1]}
in a thesis under the direction of one of the founders of time
scales calculus, B. Aulbach. Since 1992,
several other works on the subject appeared
--- see, \textrm{e.g.}, \cite{Rev[3],Rev[2],MPT}.

One important and very active subject being developed
within the theory of time scales consists in the study
of inequalities --- see \cite{Rev[4],si:ts,GronTS,SA:F:T,Gruss,czebyszew}
and references therein. To the best of our knowledge
all the integral inequalities available
in the literature of time scales are, however,
formulated using the Riemann integral on time scales.
Here we use the more general Riemann--Stieltjes
integral on time scales \cite{MPT,Rev[1]}.

After some preliminaries on the Riemann--Stieltjes integral on time scales
\cite{MPT,Rev[1]} (Section~\ref{sec:pre:not}), where we recall the main
definitions and results necessary in the sequel, we begin by generalizing
the notion of Riemann--Stieltjes delta integral for double integrals,
proving its main properties (Section~\ref{subsec:RS-delta-int}).
The main contributions of the paper are
the new dynamic inequalities for Riemann--Stieltjes
delta integrals obtained in Section~\ref{subsec:ineq}
that generalize the results of \cite{BCD},
and the two majorisation theorems of Section~\ref{subsec:mt}
that extend the results of \cite{dragomir} to the context of time scales.

We are not aware of any paper in the literature about
majorisation inequalities for Stieltjes integrals on time scales.
Our results seem to be the first in this direction.


\section{Preliminaries and Notation}
\label{sec:pre:not}

Through the text $\T$, $\T_1$, and $\T_2$ denote time scales.
Let $a,b\in\T$ and $a<b$. We distinguish $[a,b]$ as a real interval
and we define $[a, b]_{\T}:=[a,b]\cap\T$. In that
sense $[a,b]=[a,b]_{\R}$. Thus, $[a,b]_{\T}$ is a nonempty
and closed (bounded) set consisting of points from $\T$.

We recall the notion of Riemann--Stieltjes integral on a
time scale. For more we refer the reader to \cite{MPT}.
A partition of $[a, b]_{\T}$ is any finite ordered subset
\[P=\{t_0, t_1,\ldots, t_n\}\subset [a,b]_{\T}, \ \
\makebox{where} \  a=t_0<t_1<\ldots<t_n=b\, .\]
Each partition $P=\{t_0, t_1,\ldots, t_n\}$ of $[a,b]_{\T}$ decomposes it into
subintervals $[t_{i-1},t_i)_{\T}$, $i=1,2,\ldots, n$, such
that for $i\neq k$ one has $[t_{i-1},t_i)_{\T}\cap [t_{k-1},t_k)_{\T}=\emptyset$.
Each such decomposition of $[a,b]_{\T}$ into subintervals
is called a subdivision of $[a, b]_{\T}$. By $\Delta
t_i=t_i-t_{i-1}$ we denote the length of the $i$th subinterval in
the partition $P$. By $\mathcal{P}([a,b]_{\T})$
we denote the set of all partitions of $[a, b]_{\T}$.
Let $P_n$, $P_m \in\mathcal{P}([a, b]_{\T})$. If $P_n\subset P_m$
we call $P_m$ a \emph{refinement of $P_n$}.
If $P_n,P_m$ are independently chosen, then
the partition $P_n\cup P_m$ is a common refinement of $P_n$ and $P_m$.
This procedure is introduced in \cite{book:ts1}.

Let $g$ be a real-valued non-decreasing function on $[a, b]_{\T}$.
For the partition $P$ we define the set
\[g(P)=\{g(a)=g(t_0), g(t_1), \ldots,
g(t_{n-1}),g(t_n)=b\}\subset g([a, b]_{\T})\,.\]
Then, $\Delta g_i=g(t_i)-g(t_{i-1})$ is non negative and
$\sum\limits_{i=1}^n\Delta g_i=g(b)-g(a)$. Note that $g(P)$ is
a partition of $[g(a),g(b)]_{\R}=\bigcap\{J: g(P)\subset J\}$.
It is clear that even for the class of rd-continuous functions
defined on an arbitrary time scale, the image $g([a,b]_{\T})$
does not need to be a real interval
(indeed, our interval $[a, b]_{\T}$ may contain scattered points).

We now recall the definitions of lower and upper sums and
the notion of Darboux--Stieltjes sum (for more details see \cite{MPT}).
Let $f$ be a real-valued and  a bounded  function on the interval
$[a, b]_{\T}$. Let us take the partition $P=\{t_0, t_1,\ldots, t_n\}$ of $[a,b]_{\T}$.
Let $m_i=\inf_{t\in [t_{i-1},t_i)_{\T}} f(t)$ and $M_i=\sup_{t\in [t_{i-1},t_i)_{\T}} f(t)$,
$i=1,2,\ldots,n$. The upper Darboux--Stieltjes
sum of $f$ with respect to the partition $P$, denoted by $U(P,f,g)$,
is defined by $U(P,f,g)=\sum_{i=1}^nM_i\Delta g_i$ and the lower
Darboux--Stieltjes sum of $f$ with respect to the partition $P$,
denoted by $L(P,f,g)$, is defined by $L(P, f, g)=\sum_{i=1}^n m_i\Delta g_i$.

\begin{definition}[\cite{MPT}]
\label{R-S_def}
The upper Darboux--Stieltjes $\Delta$-integral from $a$ to $b$
with respect to function $g$ is defined by
$\overline{\int_a^b}f(t)\Delta g(t)
=\inf_{P\in \mathcal{P}\left([a,b]_{\T}\right)} U(P, f, g)$.
The lower Darboux--Stieltjes $\Delta$-integral from $a$ to
$b$ with respect to function $g$ is defined by
$\underline{\int_a^b}f(t)\Delta g(t)=\sup_{P\in \mathcal{P}([a, b]_{\T})} L(P, f, g)$.
If  $\overline{\int_a^b}f(t)\Delta
g(t)=\underline{\int_a^b}f(t)\Delta g(t)$, then we say that $f$ is
$\Delta$-integrable with respect to $g$ on $[a, b]_{\T}$, and the common
value of the integrals is  denoted by $\int_a^bf(t)\Delta
g(t)=\int_a^bf\Delta g$ and it is called the Riemann--Stieltjes
(or just Stieltjes) $\Delta$-integral of $f$ with respect to $g$ on
$[a, b]_{\T}$.
\end{definition}

From now on we assume that $f$ and $g$ are arbitrary real-valued
bounded functions on $[a,b]_{\T}$, where $a,b\in\T$
and $g$ is non-decreasing on $[a,b]_{\T}$.
Let us consider the partition
$P=\{t_0, t_1,\ldots, t_n\}$ of $[a,b]_{\T}$ and let
$X=\{x_1,\ldots,x_n\}$ denote an arbitrary selection
of points from $[a,b]_{\T}$ with $x_i\in [t_{i-1},t_i)_{\T}$,
$i=1,2,\ldots, n$. We define
\begin{equation}\label{eq:riem_sum}
 S_g(f, P, X)=\sum_{i=1}^nf(x_i)\left(g(t_i)-g(t_{i-1})\right)
\end{equation}
as a \emph{Riemann--Stieltjes $\Delta$-sum for $f$ with respect to $g$}.

\begin{definition}\label{def:s1}
We say that $f$ is Riemann--Stieltjes $\Delta$-integrable
with respect to $g$ and write $f\in \mathcal{S}([a,b]_{\T},g)$ if and only if
there exists a number $\mathcal{I}\in\R$ such that for every $\varepsilon>0$
there is a partition $P^*$ for which $|S_g(f,P,X)-\mathcal{I}|<\varepsilon$
for all refinements $P\supset P^*$ and all possible selections
of points $X$. If such a number exists,
it is unique, and we define $\int_a^bf\Delta g=\mathcal{I}$.
\end{definition}

Note that if $g$ is non-decreasing,
then $L(P, f, g)\leq S_g(f, P, X) \leq U(P, f, g)$ for any $P$ and $X$.
Let $\T_1$, $\T$ be time scales and $\psi:\T_1\rightarrow\T$
be a rd-continuous non-decreasing map such that for
$t_1\in [\alpha,\beta]_{\T_1}$, $a=\psi(\alpha)$, $b=\psi(\beta)$.
Then, because of the existing bijection between partitions of intervals
$[a,b]_{\T}$ and $[\alpha,\beta]_{\T_1}$ and between selections
of points from the respective intervals, the following holds:
\begin{equation*}
\int_a^bf(t)\Delta t=\int_{\alpha}^{\beta} f(\psi(t_1))\Delta_1g(\psi(t_1))\,.
\end{equation*}

The proof of Proposition~\ref{prop}
follows directly from \eqref{eq:riem_sum}
and Definition~\ref{def:s1}.

\begin{proposition}
\label{prop}
Let $g$ be non-decreasing on $[a,b]_{\T}$ and $f$
be Riemann--Stieltjes $\Delta$-integrable
with respect to $g$ on $[a,b)_{\T}$. Then,

a) $\int_a^b\Delta g(t)=g(b)-g(a)$;

b) $\int_a^bf(t)\Delta g(t)=0$ for $g$ constant;

c)  $\int_a^{\sigma(a)}f(t)\Delta g(t)=f(a)(g^{\sigma}(a)-g(a))$;

d) $\int_a^b\alpha f(t)\Delta (\beta g(t))=\alpha\beta\int_a^bf(t)\Delta g(t)$,
$\alpha, \beta \in \mathbb{R}$.
\end{proposition}
Note that if $f$ is rd-continuous and $g$ has its
$\Delta$-derivative
also as a rd-continuous function, then we can write the approximating
sum \eqref{eq:riem_sum} for $fg^{\Delta}$  with respect to the constant
function of value $1$ in the form $S_1(fg^{\Delta},P,X)=\sum_{i=1}^n
f(x_i)g^{\Delta}(x_i)\Delta t_i$. Using the mean value theorem
\cite{book:ts1}, we conclude with the following result:

\begin{theorem}[\cite{MPT}]
\label{dor:transition}
Let $a$, $b\in\T$. Suppose that $g$ is a non-decreasing
function such that $g^{\Delta}$ is continuous on $[a,b)_{\T}$
and $f$ is a real bounded function on $[a,b]_{\T}$.
Then, $f\in\mathcal{S}(g, [a,b]_{\T})$ if and only if
$fg^{\Delta}\in \mathcal{S}(g,[a,b]_{\T})$. Moreover,
\begin{equation*}
 \int_a^bf(t)\Delta g(t)=\int_a^bf(t)g^{\Delta}(t)\Delta t \, .
\end{equation*}
\end{theorem}


\section{Main Results}
\label{sec:MR}

In order to generalize the results of \cite{BCD}
to an arbitrary time scale, one needs first
to extend the Riemann--Stieltjes $\Delta$-integral
to functions of two-variables.
Properties of the double Riemann $\Delta$-integral
and for multiple Lebesgue integrals on time scales
were developed in
\cite{multiple_int,LebesgueInt:ReviewerAsked,double_int}.


\subsection{The double Riemann--Stieltjes delta integral}
\label{subsec:RS-delta-int}

Let $a$, $b \in \T_1$, $c$, $d \in \T_2$, where
$a < b$, $c < d$, and
$R=[a, b)_{\T_1}\times [c,d)_{\T_2}=\{(t,s):
t\in [a,b), s\in [c,d), t\in\T_1, s\in\T_2\}$.
Let $g_i:\T_i\rightarrow\R$, $i=1,2$, be two
non-decreasing functions on $[a,b]_{\T_1}$ and $[c,d]_{\T_2}$, respectively.
Let $f: \T_1\times\T_2\rightarrow\R$ be bounded on $R$.
Let us consider two partitions $P_1=\{t_0, t_1,\ldots, t_n\}$ of $[a,b]_{\T_1}$ and
$P_2=\{s_0, s_1,\ldots, s_k\}$ of $[c,d]_{\T_2}$ and let $X_1=\{x_1,\ldots,x_n\}$
denote an arbitrary selection of points from $[a,b]_{\T_1}$
with $x_i\in [t_{i-1},t_i)_{\T_1}$, $i=1,2,\ldots, n$.
Similarly, let $X_2=\{y_1,\ldots,y_k\}$ denote an arbitrary
selection of points from $[c,d]_{\T_2}$ with $y_j
\in [s_{j-1},s_j)_{\T_2}$, $j=1,2,\ldots, k$.
We define
\begin{equation}
\label{eq:riem_sum1}
\overline{S}_{g_1,g_2}(f, P_1, P_2, X_1, X_2)=\sum_{i=1}^n\sum_{j=1}^kf(x_i,y_j)\left(g_1(t_i)
-g_1(t_{i-1})\right)\left(g_2(s_j)-g_2(s_{j-1})\right)
\end{equation}
as the Riemann--Stieltjes $\Delta$-sum of $f$ with respect to functions
$g_1$ and $g_2$ and partitions $P_1\in \mathcal{P}([a,b]_{\T_1})$
and $P_2\in \mathcal{P}([c,d]_{\T_2})$.

\begin{definition}
We say that $f$ is Riemann--Stieltjes $\Delta$-integrable with respect
to $g_1$ and $g_2$ over $R$ if there exists a number $\mathcal{I}\in\R$ such that
for every $\varepsilon>0$ there are partitions $P_1^*$ and $P_2^*$
for which $|\overline{S}_{g_1,g_2}(f, P_1, P_2, X_1, X_2)-\mathcal{I}|<\varepsilon$ 
for all refinements
$P_1\supset P_1^*$ and $P_2\supset P_2^*$ and all possible selections
of points $X_1$ and $X_2$ corresponding to $P_1$ and $P_2$, respectively.
If such a number $\mathcal{I}$ exists, it is unique,
and we define \[\iint_R f(t,s)\Delta_{1,2}(g_1\times g_2)=\mathcal{I}\,.\]
\end{definition}

We can extend the properties of Proposition~\ref{prop}
using non-decreasing functions $g_1$ and $g_2$.
The following proposition is obtained,
\emph{mutatis mutandis},
from the proofs of similar properties of the
Riemann--Stieltjes $\Delta$-integral \cite{MPT}.

\begin{proposition}
Let $g_1$ and $g_2$ be non-decreasing functions
respectively on $[a,b]_{\T_1}$ and $[c,d]_{\T_2}$,
and let $f$ be Riemann--Stieltjes $\Delta$-integrable with respect
to $g_1$ and $g_2$ on $R=[a,b)_{\T_1}\times [c,d)_{\T_2}$. Then,

a) $\iint_R A\Delta_{1,2} \left(g_1\times g_2\right)
=A\left(g_1(b)-g_1(a)\right)\left(g_2(d)-g_2(c)\right)$, $A$ a constant;\\

b) $\iint_Rf(t,s)\Delta_{1,2} \left(g_1\times g_2\right)=0$
when $g_1$ or $g_2$ are constant;\\

c) with $b=\sigma_1(a)$ and $d=\sigma_2(c)$ one has
\[\iint_Rf(t,s)\Delta_{1,2} \left(g_1\times g_2\right)
=f(a,c)(g_1^{\sigma_1}(a)-g_1(a))(g_2^{\sigma_2}(c)-g_2(c)) \, ;\]

d) $\iint_R\alpha f(t,s)\Delta_{1,2} [\beta \left(g_1 \times g_2\right)]
=\alpha\beta\iint_Rf(t,s)\Delta_{1,2} \left(g_1\times g_2\right)$, 
$\alpha$ and $\beta$ constants.
\end{proposition}

In the classical case, \textrm{i.e.}, when $\T_1=\T_2=\R$,
the Fubini theorem is the fundamental theorem that relates
double and iterated integrals (see, \textrm{e.g.}, \cite{book:leb-st}).
The rule of iterated integration
for double Riemann $\Delta$-integrals on a rectangle was proved
in \cite[Theorem~3.10]{multiple_int}.
We extend here \cite[Theorem~3.10]{multiple_int}
to the double Riemann--Stieltjes $\Delta$-integral.

\begin{proposition}
\label{fubini}
Let $g_i:\T_i\rightarrow\R$, $i=1,2$, be two non-decreasing
functions on $[a,b]_{\T_1}$ and $[c,d]_{\T_2}$, respectively.
Let us assume that function $f: \T_1\times\T_2\rightarrow \R$
is bounded on the set $R=[a,b)_{\T_1}\times [c,d)_{\T_2}$.
Then, the existence of the integral
\begin{equation*}
 \iint_R |f| \Delta_{1,2} (g_1\times g_2)
\end{equation*}
implies the existence and the equality of the iterated integrals:
\begin{equation}
\label{th:f}
\begin{split}
\iint_R f \Delta_{1,2} (g_1\times g_2)
&= \int_a^b \left( \int_c^d f(t,s)\Delta_2 g_2(s)\right)\Delta_1 g_1(t)\\
&= \int_c^d \left( \int_a^b f(t,s)\Delta_1 g_1(t)\right)\Delta_2 g_2(s)\, .
\end{split}
\end{equation}
\end{proposition}

\begin{proof}
Let us begin noticing that if one of the functions
$g_1$ or $g_2$ is constant, then relation \eqref{th:f}
gives the truism zero equals zero.
Assume now that none of the functions
$g_1$ and $g_2$ is constant.
As it is usually done in the classical double integral calculus,
the evaluation of a double Stieltjes integral can be reduced
to the successive evaluation of two simple Stieltjes integrals.
Let $P_1\in \mathcal{P}([a,b]_{\T_1})$
and $P_2\in \mathcal{P}([c,d]_{\T_2})$ where, as in the introduction to this section, we use
$P_1=\{t_0, t_1,\ldots, t_n\}$, $P_2=\{s_0, s_1,\ldots, s_k\}$, $X_1=\{x_1,\ldots,x_n\}$,
$X_2=\{y_1,\ldots,y_k\}$, with $x_i\in [t_{i-1},t_i)_{\T_1}$, $i=1,2,\ldots, n$,
and $y_j \in [s_{j-1},s_j)_{\T_2}$, $j=1,2,\ldots, k$.
We can assume that $P_1$ is such that
$\sum_{i=1}^n\left(g_1(t_i)-g_1(t_{i-1})\right)>0$, as $g_1$ is not constant.
According to definition \eqref{eq:riem_sum1} of
Riemann--Stieltjes $\Delta$-sum we can write
\[\overline{S}_{g_1,g_2}(f, P_1, P_2, X_1, X_2)
=\sum_{i=1}^n\left(g_1(t_i)
-g_1(t_{i-1})\right)\sum_{j=1}^kf(x_i,y_j)\left(g_2(s_j)-g_2(s_{j-1})\right)\,.\]
Let us now denote by $\Phi(x_{i-1})=\int_c^df(x_{i-1},s)\Delta_2g_2(s)$ the simple Stieltjes integral
of the function $f(x_{i-1},\cdot)$ with respect to $g_2$ on the interval $[c,d]_{\T_2}$.
Using Definition~\ref{def:s1} we can write that for every
\[\overline{\varepsilon}=\frac{\varepsilon}{2\sum_{i=1}^n\left(g_1(t_i)-g_1(t_{i-1})\right)}>0\, ,\]
$\varepsilon>0$, there is a partition
$P_2^*$ such that for all refinement
$P_2\supset P_2^*$ with a selection $X_2$ we have that
\[\left|S_{g_2}(f(x_{i-1},\cdot), P_2, X_2)-\Phi(x_{i-1})\right|<  \overline{\varepsilon} \,. \]
For any partition $P_1$ of $[a,b]_{\T_1}$
with some selection $X_1$ the following holds:
\[\left|\overline{S}_{g_1,g_2}(f, P_1,P_2, X_1, X_2)
-\sum_{i=1}^n\left(g_1(t_i)-g_1(t_{i-1})\right)\Phi(x_{i-1})\right|<
\frac{\varepsilon}{2}\, . \]
It is easy to notice that the sum
$\sum_{i=1}^n\left(g_1(t_i)-g_1(t_{i-1})\right)\Phi(x_{i-1})$
represents a Riemann--Stieltjes $\Delta$-sum
for the integral $\int_a^b\Phi(t)\Delta_1g_1(t)$.
Let $\mathcal{I}=\iint_R f \Delta_{1,2} (g_1\times g_2)$.
Using again the definition in \cite{MPT}
of the simple Stieltjes delta integral on $[a,b]_{\T_1}$,
we see that for all $\varepsilon/2>0$
there is a partition $P_1^*$ such that for all refinements
$P_1\supset P_1^*$ together with all possible selections $X_1$
the following holds:
\[
\left|\overline{S}_{g_1,g_2}(f, P_1,P_2, X_1, X_2)
-\mathcal{I}\right| < \frac{\varepsilon}{2}\, .
\]
Hence,
\[\left|\mathcal{I}-\sum_{i=1}^n\left(g_1(t_i)
-g_1(t_{i-1})\right)\Phi(x_{i-1})\right|<\varepsilon \]
and $\iint_R f \Delta_{1,2} (g_1\times g_2)
= \int_a^b \left( \int_c^d f(t,s)\Delta_2 g_2(s)\right)\Delta_1 g_1(t)$.
Similarly, if we proceed in the reverse order we get the analogous formula
$\iint_R f \Delta_{1,2} (g_1\times g_2)
= \int_c^d \left( \int_a^b f(t,s)\Delta_1 g_1(t)\right)\Delta_2 g_2(s)$.
\end{proof}


\subsection{Inequalities for Riemann--Stieltjes delta integrals}
\label{subsec:ineq}

In what follows $g:\T\rightarrow\R$
is a non-decreasing function
on the interval $[a,b]_{\T}$.

\begin{proposition}\label{lem:pos}
Let $f:\T\rightarrow\R$ be Riemann--Stieltjes $\Delta$-integrable
on $[a,b]_{\T}$ with respect to a non-decreasing function $g$.
If $f$ is nonnegative on $[a,b]_{\T}$, then
\begin{equation*}
\int_a^b f(t)\Delta g(t)\geq 0\, .
\end{equation*}
\end{proposition}
\begin{proof}
If $f$ is a nonnegative function, then for any partition
$P\in \mathcal{P}([a,b]_{\T})$ we have
$\int_a^bf(t)\Delta g(t)\geq L(P,f,g)\geq 0$.
\end{proof}

\begin{corollary}\label{cor:pos}
Let $f_1,f_2:\T\rightarrow\R$ be Riemann--Stieltjes delta integrable
on $[a,b]_{\T}$ with respect to a non-decreasing function $g$.
Suppose that $f_1(t)\geq f_2(t)$ for all $t\in[a,b]_{\T}$. Then,
\begin{equation*}
\int_a^bf_1(t)\Delta g(t)\geq
\int_a^b f_2(t)\Delta g(t) \, .
\end{equation*}
\end{corollary}
\begin{proof}
The result follows immediately from Proposition~\ref{lem:pos}
and the nonnegativity of function $f(t)=f_1(t)-f_2(t)$.
\end{proof}

Similarly, we can also show the following:

\begin{proposition}
Let $R=[a,b)_{\T_1}\times[c,d)_{\T_2}$ and $f$, $f_1$, and $f_2$
be bounded functions on $R$ satisfying the inequality
$f_1(t,s)\geq f_2(t,s)$ for all $(t,s)\in R$. Then,
\begin{equation*}
\iint_R f_1 \Delta_{1,2} (g_1\times g_2)
\geq \iint_R f_2 \Delta_{1,2} (g_1\times g_2)
\end{equation*}
and
\begin{equation*}
\left|\iint_R f (t,s)\Delta_{1,2} (g_1\times g_2)\right|
\leq \iint_R |f(t,s)| \Delta_{1,2} (g_1\times g_2)\, .
\end{equation*}
\end{proposition}

\begin{proposition}
\label{lem:pos1}
Let $f:\T\rightarrow\R$ be Riemann--Stieltjes $\Delta$-integrable
on $[a,b]_{\T}$ with respect to a non-decreasing function $g$.
If $f$ is nonnegative  on $[a,b]_{\T}$, then
\begin{equation*}
F(t)=\int_a^t f(\tau)\Delta g(\tau)
\end{equation*}
is a non-decreasing function on $[a,b]_{\T}$.
\end{proposition}
\begin{proof}
If $g$ is $\Delta$-differentiable on $[a,b)_{\T}$, then
Theorem~\ref{dor:transition} states that
$$\int_a^t f(\tau)\Delta g(\tau)=\int_a^t f(\tau)g^{\Delta}(\tau)\Delta \tau \, .$$
Thus, $F^{\Delta}(t)=f(t)g^{\Delta}(t)\geq 0$
and $F$ is a non-decreasing function on $[a,b)_{\T}$. On the other hand, we can use the property
that $$\int_a^{\sigma(t)}f\Delta g=\int_a^{t}f\Delta g+f(t)(g^{\sigma}(t)-g(t))\, .$$
This means that in the case when $t$ is right-scattered then
$$F^{\Delta}(t)=\frac{f(t)(g^{\sigma}(t)-g(t))}{\mu(t)}\geq 0 \, ;$$
in the case when $t$ is right-dense then
$F^{\Delta}(t)=\lim_{s\rightarrow t}\left|\frac{\int_s^tf\Delta g}{t-s}\right|
\geq 0$. Hence, $F$ is non-decreasing.
\end{proof}

Let $I$ be an interval of real numbers and $F: I\rightarrow\R$
be a convex function on $I$. Then $F$ is continuous on $int (I)$
(the interior of $I$) and has finite left and right derivatives ($F'_+$ and $F'_-$)
at each point of $int (I)$. For a convex function $F:I\rightarrow\R$
the subdifferential of $F$ is defined as the set $\partial F$
of all extended functions $\varphi: I\rightarrow\R\cup\{\pm\infty\}$
such that $\varphi(int (I))\subset \R$ and
\begin{equation}
\label{eq:sub}
F(x)\geq F(y)+(x-y)\varphi(y), \ \mbox{for} \ x,y\in I \, .
\end{equation}

When $F$ is convex, then the set $\partial F$ is nonempty because at least
$F'_+$, $F'_-\in \partial F$.  Moreover, if $\varphi \in \partial F$
then $F'_-(x)\leq \varphi(x)\leq F'_+(x)$ for $x\in int (I)$,
and $\varphi$ is a non-decreasing function.
If $x: \T\rightarrow I\subset \R$, then the composition
$F\circ x: \T\rightarrow \R$ is a function on $\T$.

The following result is a generalization of \cite[Theorem~5]{BCD}.

\begin{theorem}
\label{th:ac5}
Let $\T$ be a time scale with $a,b\in\T$,
$F: I\rightarrow\R$ be a convex function on the real interval $I$,
and $x,y,p:[a,b]_{\T}\rightarrow I$ with $p(\cdot)$ nonnegative
on $[a,b]_{\T}$. If $\varphi\in \partial F$ and $g:[a,b]_{\T}\rightarrow I$
is a non-decreasing function on $[a,\rho(b)]_{\T}$,
then the inequality
\begin{multline}
\label{eq:ineq1}
\int_a^b p(t)F(x(t))\Delta g(t)-\int_a^b p(t)F(y(t))\Delta g(t)\\
\geq \int_a^b p(t)x(t)\varphi(y(t))\Delta g(t)
-\int_a^b p(t)y(t)\varphi(y(t))\Delta g(t)
\end{multline}
holds assuming that the Riemann--Stieltjes
$\Delta$-integrals in \eqref{eq:ineq1} exist.
\end{theorem}

\begin{proof}
For all $t\in[a,b]_{\T}$ we have $x(t)$, $y(t)\in I$. From inequality
\eqref{eq:sub} we conclude that $F(x(t))-F(y(t))\geq (x(t)-y(t))\varphi(y(t))$.
Multiplying by nonnegative values $p(t)$
and integrating with respect to the non-decreasing function $g$,
we arrive to \eqref{eq:ineq1} with the help of Corollary~\ref{cor:pos}.
\end{proof}

We can use inequality \eqref{eq:ineq1}
of Theorem~\ref{th:ac5} to prove a new
Jensen's type inequality on time scales \cite{SA:F:T}
for Riemann--Stieltjes integrals.

\begin{corollary}
\label{prop:Jensen}
Let $\T$ be a time scale with $a,b\in\T$;
$F: I\rightarrow\R$ be a convex function on $I$;
$x, p:[a,b]_{\T}\rightarrow I$ be
rd-continuous with $p(\cdot)$ nonnegative on $[a,b]_{\T}$;
and $g:[a,b]_{\T}\rightarrow I$
be non-decreasing on $[a,\rho(b)]_{\T}$.
Define $A := \int_a^b p(t)\Delta g(t)>0$.
Then,
\begin{equation*}
\frac{1}{A}\int_a^b p(t)F(x(t))\Delta g(t)
\geq F\left(\frac{1}{A}\int_a^b p(t)x(t)\Delta g(t)\right)
\end{equation*}
provided both integrals exist.
\end{corollary}
\begin{proof}
It is enough to take the constant function
$y(s)\equiv \frac{1}{A}\int_a^bp(t)x(t)\Delta g(t)$ for each
$s\in[a,b]_{\T}$, and see that $y(s)\in I$.
We do the proof for $I=[c,d]$.
For $x:[a,b]_{\T}\rightarrow I$ we have
$c p(t) \leq p(t) x(t)\leq d p(t)$. Integrating both sides
with respect to the non-decreasing function $g$ we obtain:
$A c \leq \int_a^b p(t) x(t) \Delta g(t) \leq A d$.
Hence, $c\leq y(s)\leq d$.
Taking into account inequality \eqref{eq:ineq1}
of Theorem~\ref{th:ac5} we get:
\[\frac{1}{A}\int_a^bp(t)F(x(t))\Delta g(t)
\geq F(y(s))+\varphi(y(s))\left(\frac{1}{A}
\int_a^bp(t)x(t)\Delta g(t)-y(s)\right)\, ,\]
where the right-hand side is equal to $F(y(s))$.
\end{proof}

Similarly, one can obtain a Riemann--Stieltjes Jensen's
reverse integral inequality on time scales:

\begin{corollary}
\label{cor:RSJRII}
Let $\T$ be a time scale with $a,b\in\T$;
$F: I\rightarrow\R$ be a continuous convex function
on $I$; $x, p:[a,b]_{\T}\rightarrow I$
be rd-continuous with $p(\cdot)$ nonnegative on $[a,b]_{\T}$;
and $g:[a,b]_{\T}\rightarrow I$ be non-decreasing
on $[a,\rho(b)]_{\T}$ with $A=\int_a^b p(t)\Delta g(t)>0$.
If $\varphi \in \partial F$ and the Riemann--Stieltjes $\Delta$-integrals
$\int_a^b p(t)y(t)\varphi(y(t))\Delta g(t)$ and
$\int_a^b p(t)\varphi(y(t))\Delta g(t)$ exist, then
\begin{equation*}
\begin{split}
0 &\leq \frac{1}{A}\int_a^b p(t)F(y(t))\Delta g(t)
- F\left(\frac{1}{A}\int_a^b p(t)y(t)\Delta g(t)\right)\\
&\leq \frac{1}{A}\left(\int_a^bp(t)y(t)\varphi(y(t))\Delta g(t)
-\frac{1}{A}\int_a^b p(t)y(t)\Delta g(t) \cdot \int_a^bp(t)\varphi(y(t))\Delta g(t)\right)\, .
\end{split}
\end{equation*}
\end{corollary}

\begin{remark}
Corollary~\ref{cor:RSJRII} coincides with \cite[Corollary~2]{BCD}
in the particular case when $\T = \R$.
\end{remark}

Using the Riemann--Stieltjes double integral we can prove an inequality
of \v{C}eby\v{s}ev's type on time scales.
The inequality \eqref{ieq_prop16} of Proposition~\ref{prop:ceb1}
is motivated by the \v{C}eby\v{s}ev's inequality on time scales
proved in \cite{czebyszew}.

\begin{proposition}
\label{prop:ceb}
Suppose that $p\in C_{rd}\left([a,b]_{\T}\right)$ with
$p(t)\geq 0$ for all $t\in[a,b]_{\T}$,
and let $g:[a,b]_{\T} \rightarrow \R$
be non-decreasing on $[a,\rho(b)]_{\T}$.
Let $f_1$, $f_2\in C_{rd}\left([a,b]_{\T}\right)$
be similarly (oppositely) ordered, that is,
for all $t,s\in [a,b]_{\T}$
\begin{equation*}
\left(f_1(t)-f_1(s)\right)\left(f_2(t)-f_2(s)\right)\geq 0 \, (\leq 0) \, .
\end{equation*}
Then,
\begin{equation}
\label{eq:ceb}
\int_a^b\int_a^b p(t)p(s)\left(f_1(t)-f_1(s)\right)\left(f_2(t)
-f_2(s)\right)\Delta g(t)\Delta g(s) \geq 0 \, (\leq 0) \,.
\end{equation}
\end{proposition}

\begin{proof}
Follows from Proposition~\ref{lem:pos}.
\end{proof}

\begin{proposition}
\label{prop:ceb1}
Suppose that $p\in C_{rd}\left([a,b]_{\T}\right)$ with $p(t)\geq 0$
for all $t\in[a,b]_{\T}$, and let $g:[a,b]_{\T} \rightarrow \R$
be non-decreasing on $[a,\rho(b)]_{\T}$.
Let $f_1$, $f_2\in C_{rd}\left([a,b]_{\T}\right)$
be similarly (oppositely) ordered.
Then,
\begin{equation}
\label{ieq_prop16}
\int_a^b p(t)\Delta g(t)\int_a^bp(t)f_1(t)f_2(t)\Delta g(t)\geq (\leq)
\int_a^b p(t)f_1(t)\Delta g(t)\int_a^bp(t)f_2(t)\Delta g(t) \,.
\end{equation}
\end{proposition}

\begin{proof}
We need to rewrite inequality \eqref{eq:ceb} as \eqref{ieq_prop16}.
Because $p$ is a rd-continuous
function on the interval $[a,b]_{\T}$
and $g$ is non-decreasing on $[a,\rho(b)]_{\T}$ (see \cite{MPT}),
function $p$ is Riemann--Stiejtles $\Delta$-integrable
with respect to $g$. Then,
\begin{equation*}
\begin{split}
\int_a^b &\int_a^bp(t)p(s)\left(f_1(t)-f_1(s)\right)\left(f_2(t)-f_2(s)\right)\Delta g(t)\Delta g(s)\\
&= \int_a^bp(s)\int_a^b \left(p(t)f_1(t)f_2(t)-p(t)f_1(t)f_2(s)-p(t)f_1(s)f_2(t)\right.\\
&\qquad\qquad \left. + p(t)f_1(s)f_2(s)\right)\Delta g(t)\Delta g(s)\\
&=  2\left(\int_a^b p(t)\Delta g(s)\int_a^b p(t)f_1(t)f_2(t)\Delta g(t)\right. \\
&\qquad\qquad \left. - \int_a^b p(t)f_1(t)\Delta g(s)\int_a^bp(t)f_2(t)\Delta g(t)\right)\geq 0
\end{split}
\end{equation*}
and the result is proved.
\end{proof}

Corollary~\ref{inverse} gives a Winckler-type formula
for the delta Riemann--Stieltjes integral on time scales.
In the particular case $g(t)=t$ one obtains the result in \cite{czebyszew};
in the case $g(t)=t$ and $\T=\N$ we can easily
obtain the classical Winckler formula:
if $a = \left(a_1,\ldots,a_n\right)$
and $b = \left(b_1,\ldots,b_n\right)$
are similarly (oppositely) ordered, then
\begin{equation*}
\sum_{i=1}^n p_i \sum_{i=1}^n a_i b_i
\geq (\leq) \sum_{i=1}^np_i a_i \sum_{i=1}^np_ib_i \,.
\end{equation*}

\begin{corollary}
\label{inverse}
Let $p\in C_{rd}\left([a,b]_{\T}\right)$ with
$p(t)\geq 0$ for all $t\in[a,b]_{\T}$
and let $g:[a,b]_{\T} \rightarrow \R$
be non-decreasing on $[a,\rho(b)]_{\T}$.
If $f$ and $1/f \in C_{rd}\left([a,b]_{\T}\right)$, then
\begin{equation}
\label{eq:ceb2}
\left(\int_a^b p(t)\Delta g(t)\right)^2
\geq \int_a^b p(t)f(t)\Delta g(t)\int_a^b\frac{p(t)\Delta g(t)}{f(t)} \,.
\end{equation}
\end{corollary}

\begin{proof}
It is enough to take $f_1=f$ and $f_2=\pm 1/f$
in Proposition~\ref{prop:ceb1}. Indeed, from the assumption
that $f_1$, $f_2\in C_{rd}\left([a,b]_{\T}\right)$
it follows that $f_1(t)f_2(t)=\pm 1$
for each $t\in[a,b]_{\T}$.
Since $f_1$ and $f_2$ are obviously similarly or oppositely ordered,
we end up with inequality \eqref{eq:ceb2}.
\end{proof}

From Corollary~\ref{inverse} we can obtain other
Winckler formulas by choosing different time scales
and different non-decreasing functions $g$ on $\T$:

\begin{example}
Let $\T=\overline{q^{\Z}}$, $q>1$, and $g(t)=t^2$.
Choose $a=0\in\T$ and $b=1\in\T$. We consider the integral
$\int_0^1 p(t)f_1(t)f_2(t)\Delta g(t)$ on this time scale.
The $q$-scale integral is in this case represented by an infinite series:
$$\int_0^1p(t)\Delta g(t) =\sum_{k=1}^{+\infty}p(q^{-k})  q^{-k}(q+1) \, .$$
Let us take $p(t)=t$ and, analogously as in Corollary~\ref{inverse},
consider similarly ordered functions $f_1$ and $f_2$ on $[0,1]_{\T}$
with $f_1(t)f_2(t)=1$. It follows that
$\left(\int_0^1p(t)\Delta g(t)\right)^2 =\frac{1}{(q-1)^2}$ while
$$\int_0^1 p(t)f_1(t)\Delta g(t)\int_0^1p(t)f_2(t)\Delta g(t)=(q+1)^2\sum_{k=1}^{+\infty}
q^{-2k}f_1(q^{-2k})\sum_{k=1}^{+\infty} q^{-2k}f_2(q^{-2k}) \, .$$
Hence,
\[\sum_{k=1}^{+\infty} q^{-2k}f(q^{-2k})\sum_{k=1}^{+\infty}
\frac{q^{-2k}}{f(q^{-2k})}\leq \frac{1}{(q^2-1)^2}\, ,\]
where $f=f_1$.
\end{example}


\subsection{Majorisation theorems}
\label{subsec:mt}

We now extend some majorisation type results
from \cite{BCD,dragomir}.

\begin{theorem}
\label{majorization}
Let $\T$ be a time scale with $a,b\in\T$; functions
$x, y, p, g:[a,b]_{\T}\rightarrow I\subset \R$
be rd-continuous on $[a,b]_{\T}$ with $g$ non-decreasing
and $p$ bounded. Additionally, let $F:I\rightarrow\R$
be a continuous convex function on $I$.
If $y$ and $x-y$ are both non-decreasing or non-increasing and
\begin{equation}
\label{*}
\int_a^bp(t)y(t)\Delta g(t)=\int_a^bp(t)x(t)\Delta g(t) \, ,
\end{equation}
then
\begin{equation}
\label{maj1}
\int_a^bp(t)F(y(t))\Delta g(t)\leq \int_a^bp(t)F(x(t))\Delta g(t)\, .
\end{equation}
\end{theorem}

\begin{proof}
The rd-continuity assumptions imply the existence of all integrals
in \eqref{*} and \eqref{maj1}.
Moreover, if $\varphi\in \partial F$,
then both $\varphi$ and $\varphi \circ y$ are non-decreasing
on $[a,b]_{\T}$. Since $p(\cdot)$ is bounded on $[a,b]_{\T}$,
the rd-continuity of $g$ implies the existence
of the Riemann--Stieltjes $\Delta$-integral
$\int_a^bp(t)[x(t)-y(t)]\varphi(y(t))\Delta g(t)$.
Since $g$ is non-decreasing on $[a,b]_{\T}$, then
\eqref{eq:ineq1} implies that
\begin{equation}\label{**}
 \int_a^b p(t)F(x(t))\Delta g(t)-\int_a^b p(t)F(y(t))\Delta g(t)
 \geq \int_a^b p(t)[x(t)-y(t)]\varphi(y(t))\Delta g(t) \, .
\end{equation}
Taking $f_1(t)=x(t)-y(t)$
and $f_2(t)=\varphi(y(t))$ in inequality \eqref{ieq_prop16}
and noting that $f_1$ and $f_2$ are similarly ordered, we obtain:
\begin{multline*}
\int_a^b p(t)\Delta g(t)\int_a^bp(t)[x(t)-y(t)]\varphi(y(t))\Delta g(t)\\
\geq \int_a^b p(t)[x(t)-y(t)]\Delta g(t)\int_a^bp(t)\varphi(y(t))\Delta g(t) \, .
\end{multline*}
Equality \eqref{*} implies that $p(t)[x(t)-y(t)]=0$, so
\[\int_a^b p(t)\Delta g(t)\int_a^bp(t)[x(t)-y(t)]\varphi(y(t))\Delta g(t) \geq 0.\]
From Proposition~\ref{lem:pos} it follows that $\int_a^b p(t)\Delta g(t) \geq 0$. Thus,
\[\int_a^bp(t)[x(t)-y(t)]\varphi(y(t))\Delta g(t) \geq 0 \, . \]
Inequality \eqref{maj1} follows from \eqref{**}.
\end{proof}

\begin{theorem}
Let $\T$ be a time scale with $a,b\in\T$; functions
$x, y, p, g:[a,b]_{\T}\rightarrow I\subset \R$
be rd-continuous on $[a,b]_{\T}$ with $g$ non-decreasing
and $p$ bounded and nonnegative. Additionally, let $F:I\rightarrow\R$
be a non-decreasing continuous and convex function on $I$.
If $y$ and $x-y$ are both non-decreasing or non-increasing and
\begin{equation}
\label{*'}
\int_a^bp(t)y(t)\Delta g(t) \leq \int_a^bp(t)x(t)\Delta g(t) \, ,
\end{equation}
then \eqref{maj1} holds true.
\end{theorem}

\begin{proof}
The integrals $\int_a^b p(t)[x(t)-y(t)]\Delta g(t)$
and $\int_a^b p(t)\varphi(y(t))\Delta g(t)$ that appear
in the proof of Theorem~\ref{majorization} are nonnegative
because of \eqref{*'} and the monotonicity of $F$ and nonnegativeness of $p$.
Thus, \[\int_a^bp(t)[x(t)-y(t)]\varphi(y(t))\Delta g(t) \geq 0 \, .\]
\end{proof}


\section*{Acknowledgements}

Dorota Mozyrska was partially supported by the
Bia{\l}ystok University of Technology grant S/WI/1/08;
Ewa Paw\l uszewicz and Delfim F. M. Torres
by the {\it Centre for Research on
Optimization and Control} (CEOC) from the {\it Portuguese
Foundation for Science and Technology} (FCT), cofinanced by the
European Community fund FEDER/POCI 2010.



\end{document}